\documentclass[12pt]{amsart}

\usepackage[color,notcite]{ }
\def \log{\ln}

\usepackage{hyperref}
\usepackage[headinclude,DIV13]{typearea}
\areaset{15.1cm}{25.0cm}
\usepackage{txfonts,amssymb,amsmath,amsthm,bbm}

\usepackage{graphicx, psfrag}
\usepackage{xcolor}
\usepackage{subfigure}

\newtheorem{theorem}{Theorem}[section]
\newtheorem{lemma}[theorem]{Lemma}
\newtheorem{proposition}[theorem]{Proposition}

\newtheorem{conjecture}[theorem]{Conjecture}

\definecolor{bbm}{RGB}{51,153,0}
\definecolor{above}{RGB}{128,0,128}
\definecolor{below}{RGB}{102,0,204}
\definecolor{cascade}{RGB}{204,0,0}
\definecolor{iid}{RGB}{153,51,0}

\theoremstyle{remark}
\newtheorem*{remark}{Remark}

\def\paragraph#1{\noindent \textbf{#1}}

\numberwithin{equation}{section}

\def\d{\mathrm{d}}

\def\<{\langle}
\def\>{\rangle}

\def\a{\alpha}

\def\e{\epsilon}

\def\s{\sigma}

\def\D{\Delta}

\def\N{{\Bbb N}}  
\def\P{{\Bbb P}}

\def\E{{\Bbb E}}

\let\cal=\mathcal

\def\FF{{\cal F}}

 \def \s {{\sigma}}

 \def \D {{\Delta}}

 \def \d {{\delta}}
 \def \a {{\alpha}}

 \def \ba {\begin{array}}
 \def \ea {\end{array}}

 \newcommand{\be}{\begin{equation}}
 \newcommand{\ee}{\end{equation}}

\newcommand{\bea}{\begin{eqnarray}}
 \newcommand{\eea}{\end{eqnarray}}
\def\TH(#1){\label{#1}}\def\thv(#1){\ref{#1}}
\def\Eq(#1){\label{#1}}\def\eqv(#1){(\ref{#1})}

\def\sfrac#1#2{{\textstyle{#1\over #2}}}

 \def \1{\mathbbm{1}}

\newcommand{\Leb}{{\rm Leb}}

\begin{document}

 \title[High points of the randomized RZF ]
{High points of a random model of the Riemann-zeta function and Gaussian multiplicative chaos}
\author[L.-P. Arguin]{Louis-Pierre Arguin}
 \address{L.-P. Arguin\\ Department of Mathematics\\
Baruch College and Graduate Center City University of New York\\
New York, New York 10010\\
USA}
\thanks{The research of L.-P. A. is supported in part by NSF CAREER~DMS-1653602. 
}

\email{louis-pierre.arguin@baruch.cuny.edu}

\author[L. Hartung]{Lisa Hartung}
 \address{L. Hartung\\ Institut f\"ur  Mathematik\\
Johannes Gutenberg-Universit\"at Mainz\\
Staudingerweg 9\\
55099 Mainz\\
Germany  }
\email{lhartung@uni-mainz.de}
\author[N. Kistler]{Nicola Kistler}
 \address{N.Kistler\\Institut f\"ur  Mathematik\\
Goethe-Universit\"at Frankfurt \\ Robert-Mayer-Str. 10\\
60325 Frankfurt\\ Germany }
\email{kistler@math.uni-frankfurt.de}

\subjclass[2000]{60J80, 60G70, 82B44} \keywords{Riemann-Zeta function, high points, Gaussian multiplicative chaos, extreme values } 

\date{\today}

 \maketitle 
  \begin{abstract}  
 We study the total mass of high points in a random model for the Riemann-Zeta function. We consider the same model as in \cite{harper13, ABH17}, and build on the convergence to 'Gaussian' multiplicative chaos proved in \cite{SW16}. We show that the total mass of points which are a linear order below the maximum divided by their expectation converges almost surely to the Gaussian multiplicative chaos of the approximating Gaussian process times a random function. We use the second moment method together with a branching approximation to establish this convergence.
  \end{abstract}

\section{Introduction}

 \subsection{The model}

Let $\mathcal{P}$ denote the set of all prime numbers. Let $(\theta_p)_{p\in\mathcal{P}}$ be independent identically distributed random variables, being uniformly distributed on $[0,2\pi]$. For $N\in\mathbb{N}$, a good model for the large values of the logarithm of the Riemann-zeta function on a typical interval of length $1$ of the critical line as proposed in \cite{harper13}  is
 \be\Eq(model.1)
 X_N(x)=\sum_{j=1}^N \frac{1}{\sqrt{p_j}} \left(\cos(x\log p_j) \cos( \theta_{p_j})+\sin(x\log p_j) \sin( \theta_{p_j})\right) \quad x\in[0,1]\ .
 \ee
 
By Theorem 7 in \cite{SW16},  the process $X_N$ can be well approximated by a log-correlated Gaussian field $G_N(x),x\in[0,1]$. Namely, take
 \be\Eq(model.2)
 G_N(x)= \sum_{j=1}^N \frac{1}{2\sqrt{p_j}} \left(W_j^{(1)}\cos(x\log p_j)  +W_j^{(2)}\sin(x\log p_j)  \right),
 \ee
where  $(W_j^{(i)})_{j\in \mathbb{N},i\in \{1,2\}}$ are i.i.d.\ standard normal distributed. It is shown in \cite{SW16} that
\be\Eq(model.3)
X_N(x)-G_N(x)\equiv E_N(x), \quad x\in[0,1],
\ee
where $E_N(x)$ converges almost surely uniformly to a  random function $E(x)$. Moreover, the error $E_N(x)$ has uniform exponential moments
\be
\Eq(model.4)
\E\left(e^{\lambda \sup_{N\geq 1,x\in[0,1]}E_N(x)}\right)<\infty,
\ee
 where $\E$ denotes expectation with respect to the $\theta_p$'s.

Some of the behavior of the large values of the process $X_N(x)$, $x\in[0,1]$ is captured by the random measure
\be
M_{\alpha,N}(dx)=\frac{e^{\alpha X_N(x)}}{\mathbb{E}e^{\alpha X_N(x)} }dx\ .
\ee
By the independence of the $\theta_p$'s, it is not hard to see that $M_{\alpha,N}$ converges almost surely as $N\to\infty$.
By Theorem 4 in \cite{SW16}, the almost sure weak limit of $M_{\alpha,N}(dx)$ is non-trivial for $0<\alpha<2$. 
We denote the limit of the total mass by $M_\alpha$
\be\Eq(model.5)
M_\a= \lim_{N\to\infty} \int_0^1 M_{\alpha,N}(dx) \ \ a.s. 
\ee
 For log-correlated Gaussian field the analogous limiting measure is called Gaussian multiplicative chaos and $M_\a$ corresponds to the total mass of the limiting measure. For Gaussian multiplicative chaos it was first proven by \cite{K85} that the limit is nontrivial for small $\a$ and was recently revisited (see \cite{RoV10, RhV14}). Note that in our case the limit of $M_{\a,N}(dx)$ is almost a Gaussian multiplicative measure (see \cite{SW16}). The connection between the Riemann-zeta function and Gaussian multiplicative chaos has been further analysed in \cite{SW16n}.

The fact that the Riemann-zeta function (or a random model of it) can be well approximated by a log-correlated field have recently been used to study the extremes on a random interval \cite{CNN17, N18, ABH17}.
\subsection{Main result}

Consider the Lebesgue measure of $\alpha$-high points:
\be\Eq(result.1)
W_{\alpha,N}=\Leb\{X_N(x)>\frac{\alpha}{2} \log\log N\}
\ee
The main result of this note is to relate the limit $M_\alpha$ to the Lebesgue measure of high points building on the ideas of \cite{GKS18}:
\begin{theorem}\TH(THM.1)
For any $0<\a<2$ and $M_{\alpha}$ as in \eqv(model.5), we have
\be\Eq(result.2)
 \frac{W_{\a,N}}{\E\left(W_{\a,N}\right)} \to M_\a,  
\ee
in probability as $N\to \infty$.
\end{theorem}

It was proved in \cite{ABH17} that the maximum of $X_N(x)$ on $[0,1]$ is $\log\log N-(3/4\pm \e) \log\log\log N$ with large probability. 
In view of this and of Theorem \thv(THM.1), it is not surprising to see that the $M_\alpha$ is non-trivial for $\alpha<2$.
The critical case where  $\alpha\to 2$ is interesting as it is related to the fluctuations of the maximum of $X_N$. 
It is reasonable to expect that our approach can be adapted to the method of \cite{Dup14a} to prove the critical case.
Another upshot of the proof is that it highlights the fact that $M_\alpha$ depends on small primes, cf. Lemma \thv(Lem.first2).

The problem for the Riemann-zeta function is trickier.  We expect that the equivalent of Theorem \thv(THM.1) still holds:
\begin{conjecture}
Let $\tau$ be a uniform random variable on $[T,2T]$. 
Let $W_{\alpha, T}=\Leb\{h\in[0,1]:\log |\zeta(1/2+i(\tau+h))|>\frac{\alpha}{2} \log\log T\}$. Then we have
$$
\lim_{T\to\infty} 
\frac{W_{\a, T}}{\E[W_{\a, T}]}
=\lim_{T\to\infty} \frac{\int_0^1|\zeta(1/2+i(\tau+h))|^\alpha}{\E[|\zeta(1/2+i\tau|^\alpha]}\quad a.s.
$$
\end{conjecture}
This would be consistent with the conjecture of Fyodorov \& Keating for the Lebesgue measure of high points, see Section 2.5 in \cite{FK14} 
One issue is that it is not obvious that a result akin to Equation \eqv(model.3) holds, mainly because of the singularities of $\log \zeta$ at the zeros.
One way around this would be to restrict to Gaussian comparison to one-point and two-point large deviation estimates. 
This seems doable in view of Lemmas \thv(lem.sec.2) and \thv(Lem.sec.3) and the Gaussian comparison theorem proved for the zeta function in \cite{ABBRS19}.

\subsection{Outline of the proof} The proof of Theorem \thv(THM.1) is based on a first and a second moment estimate and follow the global strategy proposed in \cite{GKS18} for branching Brownian motion. First, we prove convergence of a conditional first moment to the desired limiting object in Lemma \thv(Lem.first2). Its proof builds on results on the Gaussian comparison and convergence to Gaussian multiplicative chaos established in \cite{SW16}. Next, a localisation result is established in Lemma \thv(Lem.first.3). Finally, we turn to the proof of Proposition \thv(prop.sec) which is based on a second moment computation. We use a branching approximation similar to the one employed in \cite{ABH17}.
Using the obtained first and second moment estimates we are finally in the position to prove Theorem \thv(THM.1).

 \noindent\textbf{Acknowledgements.} Lisa Hartung and Nicola Kistler thank the Rhein-Main Stochastic group for creating an interactive research environment leading to this article.

\section{First moment estimates}
For $R\leq N$, we define $\FF_R$ to be the $\s$-algebra generated by $(\theta_{p})_{p\leq R}$.
 We will often condition on $\FF_R$ to fix the dependence on the small primes. The variance of $G_N(x)-G_R(x)$, $x\in [0,1]$ is by definition
 \be\Eq(Eqn.var)
\sigma_R^2(N) \equiv{\rm Var} (G_N(x)-G_R(x))=\frac{1}{2}\sum_{R<p\leq N} p^{-1}
 \ee
The prime number theorem, see e.g. \cite{montgomery-vaughan}, implies that the density of the primes goes like $(\log p)^{-1}$. More precisely, we have
 \be\Eq(first.11)
 \sigma_R^2(N)=\left\vert\sum_{R<p\leq N}^N p^{-1} -\frac{1}{2}(\log\log N-\log \log R) \right\vert=o(1)\ \ \text{as $N\to\infty$ and $R\to\infty$. }
 \ee

 It turns out that the non-trivial contribution to Theorem \thv(THM.1) comes from the small primes. 
  \begin{lemma}\TH(Lem.first2)
For $W_{\alpha,N}$ as in \eqv(result.1), we have for $0<\alpha<2$
\be\Eq(first.7)
 \lim_{R\to\infty}  \lim_{N\to\infty}\frac{ \E\left(W_{\alpha,N}\vert \FF_R\right)}{\E\left(W_{\alpha,N}\right)}
=M_{\a}  \quad \mbox{a.s.}
  \ee
\end{lemma}
\begin{proof}
We start by computing $\E\left(W_{\alpha,N}\vert \FF_R\right)$. Using  Fubini's Theorem we can write the left-hand side of \eqv(first.7) as 
\bea\Eq(first.8)
&&\int_0^1 \P\left(X_N(x)>\frac{\alpha}{2} \log \log N\Big\vert \FF_R\right) dx\\
&=& \int_0^1 \P\left( G_N(x)-G_R(x) +(E_N(x)-E_R(x))> \frac{\alpha}{2} \log \log N - G_R(x)-E_R(x)\Big|\FF_R\right)dx,\nonumber
\eea
where we used \eqv(model.3). Moreover, again for each $\e>0$ there is $R_0$ such that for all $R\geq R_0$ $|E_{R}(x)-E_{N}(x)|<\e$ almost surely and uniformly in $x$. Hence, we can again upper bound \eqv(first.8) by 
\be\Eq(first.9)
  \int_0^1 \P\left( G_N(x)-G_R(x)> \frac{\alpha}{2} \log \log N - G_R(x)-E_R(x)-\e)\Big|\FF_R\right)dx,
\ee
 and a corresponding lower by replacing $\e$ by $-\e$. Next, observe that by definition of $X_N(x)$ and $E_N(x)$, $G_N(x)-G_R(x)$ are independent of $\FF_R$. We have that the probability in \eqv(first.9) is bounded from above by
 \bea \Eq(first.10)
 &&\frac{\s_R(N)}{\sqrt{2\pi}(\alpha \log \log N - G_R(x)-E_R(x)-\e)}
 \exp\left(-\frac{(\frac{\alpha}{2} \log \log N - G_R(x)-E_R(x)-\e)^2}{ 2\sigma_R^2(N)}\right)\nonumber\\
 &&= \frac{\sigma_R(N)}{\sqrt{2\pi}(\alpha \log \log N )}
 \exp\left(- \frac{\a^2 (\log\log N)^2}{8\sigma_R^2(N)}+\alpha (E_R(x)+G_R(x)+\e)\right) (1+o(1)),
 \eea

 Next, we turn to $\E\left(W_{\alpha,N}\right)$.
 We have that 
 \bea\Eq(first.12)
 &&\E\left(W_{\alpha,N}\right)=\mathbb{E}\left(  \E\left(W_{\alpha,N}\vert \FF_R\right)\right)
 \leq\frac{\s_r(N)}{\sqrt{2\pi}(\frac{\alpha}{2}\log \log N)  }
 \exp\left(- \frac{\a^2 (\log\log N)^2}{8\s_R^2(N)}\right)\nonumber\\
 &&\hspace{5cm}\times\int_0^1 \mathbb{E}\left(\exp\left(\alpha (E_R(x)+G_R(x)-\e)\right)\right)dx (1+o(1))
 \eea
 A corresponding lower bound we obtain by replacing $\e$ by $-\e$.
 Taking the quotient of \eqv(first.10) and and \eqv(first.12) and integrating with respect to $x$ we get
 \bea\Eq(first.12.2)
 && \hspace{-2cm}\frac{\int_0^1 \exp\left(\alpha (E_R(x)+G_R(x)+\e)\right)}{\int_0^1\mathbb{E}\left(\exp\left(\alpha (E_R(x)+G_R(x)-\e)\right)\right)dx} (1+o(1))\nonumber\\
\qquad&\leq&\frac{ \E\left(W_{\alpha,N}\vert \FF_R\right)}{\E\left(W_{\alpha,N}\right)}
 \leq \frac{\int_0^1 \exp\left(\alpha (E_R(x)+G_R(x)-\e)\right)}{\int_0^1\mathbb{E}\left(\exp\left(\alpha (E_R(x)+G_R(x)+\e)\right)\right)dx}(1+o(1)),
 \eea
 Pulling the terms involving $\e$ out of the integral and noting the normalization of $M_{\a,R}$ is chosen such that  $\E M_{\a,R}=1 $ and noting that
 \be\Eq(first.13)
 \E\left(  \frac{\int_0^1 \exp\left(\alpha (E_R(x)+G_R(x))\right)dx}{\int_0^1\mathbb{E}\left(\exp\left(\alpha (E_R(x)+G_R(x))\right)\right)dx}\right)=1,
 \ee
  we can rewrite \eqv(first.12) as
  \be\Eq(first.14)
    M_{\a,R}e^{2\a\e}(1+o(1))
\leq\frac{ \E\left(W_{\alpha,N}\vert \FF_R\right)}{\E\left(W_{\alpha,N}\right)}\leq M_{\a,R}e^{-2\a\e}(1+o(1)).
  \ee
 Note that \eqv(first.14) holds for all $\e>0$. When  taking $N,R\uparrow \infty $ $M_{\a,R}$ converges a.s. to $M_\alpha$ hence we have a.s.
 \be\Eq(first.15)
    M_{\a}e^{2\a\e}(1+o(1))
\leq\liminf_{N,r\to\infty}\frac{ \E\left(W_{\alpha,N}\vert \FF_R\right)}{\E\left(W_{\alpha,N}\right)}
\leq \limsup_{N,r\to\infty}\frac{ \E\left(W_{\alpha,N}\vert \FF_R\right)}{\E\left(W_{\alpha,N}\right)}
\leq M_{\a}e^{-2\a\e}(1+o(1)).
  \ee
As \eqv(first.15) does not depend on $r$ and $N$ anymore, we can take the limit as $\e\to 0$ and obtain 
 \be\Eq(first.16)
 \lim_{R\to\infty} \lim_{N\to\infty}\frac{ \E\left(W_{\alpha,N}\vert \FF_R\right)}{\E\left(W_{\alpha,N}\right)}
=M_{\a} \ .
 \ee
 \end{proof}
 
Next, we want to control
\be\Eq(first.17)
W_{\a,N}^{>}=\Leb\{x\in[0,1]: X_N(x)\geq \alpha\log\log N; \exists k\in\left[R ,N\right]:\; X_k(x)>(\alpha+\epsilon)\log\log k\}\ .
\ee
The idea is that, at high points, the value $X_N(x)$ is most likely shared equally by the increments as defined in \eqv(second.1) below.

\begin{lemma}\TH(Lem.first.3)
For all $\epsilon>0$ there exists $R_0$ such that for all $R=o(N)$ and $R,N>R_0$ such that for all  $c>0$
\be\Eq(first.18)
\P\left(W_{\a,N}^{>}>c\E W_{\a,N}\right)\leq e^{-\e r},
\ee
where $r=\ln\ln R $.
\end{lemma}
\begin{proof}
We want to use Markov's inequality to bound the probability on \eqv(first.18). Hence, we need to bound $\E W_{\a,N}^{>}$ from above. First, we bound $\E\left( W_{\a,N}^{>}\vert\FF_R\right)
$ from above by
\bea\Eq(first.19)
&&\int_0^1  \P\Big (\left\{G_N(x)-G_R(x)\geq \alpha\log\log N-X_R(x)-\e'\right\} \\
&&\hspace{1cm}\cap\left\{\exists K\in\left[R,N\right]:\; G_K(x)-G_R(x)>(\alpha+\epsilon)\log\log K-X_R(x)-\e'\right\}\Big\vert \FF_R\Big),\nonumber
\eea
where we used \eqv(model.3) and the fact that $E_R(x)$ converges a.s.~uniformly to a continuous function $E(x)$. Hence, for all $\epsilon'>0$ there is $R_0$ such that for all $K\geq R_0$ and all $x$ we have $\vert E_K(x)-E_R(x)\vert <\e'$.

Similarly as in \eqv(Eqn.var), the variable $G_K(x)-G_R(x)$ is Gaussian with mean $0$ and variance  
$$
\sigma_R^2(K)=\frac{1}{2}\sum_{R<p\leq K} p^{-1}\ .
$$
 Let
\be\Eq(first.20)
B_K(x)= G_K(x)-G_R(x)-\frac{\sigma_R^2(K)}{\sigma_R^2(N)}(G_N(x)-G_R(x)),
\ee
then $(B_K(x))_{K=1}^N$ are points on a time-changed brownian bridge from zero to zero in time $\sigma_R(N)^2$. 
As a Brownian bridge is independent from its endpoint, Equation \eqv(first.19) is equal to
\bea\Eq(first.21)
&&\int_0^1 \int_{\frac{\alpha}{2} \log \log N-X_R(x)-\e'}^{\infty}\P(G_N(x)-G_R(x)\in dy)\\
&& \qquad \times \P\left(\exists K\in\left[R ,N\right]:\; B_K(x)>\frac{\alpha+\epsilon}{2}\log\log K-X_R(x)-\e'-\frac{\sigma_R^2(K)}{\sigma_R^2(N)} y\ \Big\vert \FF_R \right)dx \nonumber
\eea
  as $\vert \sum_{p\leq K} p^{-1}-\log\log K\vert<C$. Let $r=\log \log R$ and $n=\ln\ln N$.   
  Next, let us control the probability that $X_{R}(x)$ is too large. 
  \be\Eq(Lisa.103)
  \P\left(X_{R}(x)\geq \frac{\e r}{3}\right)\leq \P\left(G_{R}(x)\geq \frac{\e r}{4}\right) + \P\left(E_{R}(x)\geq \frac{\e r}{12}\right).
  \ee 
  The second probability in \eqv(Lisa.103) is bounded by $Ce^{-\frac{\e r}{12}}$ by \eqv(model.4). For the first probability in \eqv(Lisa.103) is bounded by $Ce^{-\frac{\e r}{32}}$ by Gaussian tail asymptotics and the variance estimate for $G_R(x)$ for $r$ large enough and uniformly in $x$.
  On the event that $\{X_{R}(x)\leq \frac{\e r}{3}\}$ we can bound 
the second probability in \eqv(first.21)  from above by
\be\Eq(first.22)
 \P\left(\exists s\in\left[0,\sigma_{r}^2(N)\right]:\; b(s)> (\a+\epsilon)\left((s+\s_0(R)^2)  -C\right)-\frac{\e r}{3}-\e'-\frac{s}{\sigma_{r}^2(N)} y\right),
\ee
where $b(s)$ is a Brownian bridge from zero to zero in time $\sigma_{r}^2(N)$.
Consider the line $l$ from $(0,\frac{\e}{6}r)$ to $\left(\sigma_{r}^2(N),(\a+\e)(n/2-C)-\e'-y\right)$. One checks that $l(s)\leq (\a+\epsilon)\left((s+\s_0(R)^2)  -C\right)-\frac{\e r}{3}-\e'-\frac{s}{\sigma_{r}^2(N)} y$ for all $r$ large enough.
The probability of Brownian bridge to stay under a linear function is well known, see e.g.,   Lemma 2.2 in \cite{B_C},  
\be\Eq(Lisa.104)
\P\left(\exists s\in\left[0,\sigma_{r}^2(N)\right]:\; b(s)> l(s)\right)
=\exp\left(-2\frac{l(0)l\left(\sigma_{r}^2(N)\right)}{\sigma_{r}^2(N)}\right)
\ee
Hence, on the event  we can bound the expectation of \eqv(first.21) by
\be\Eq(Lisa.105)
\int_0^1 \E\int_{\frac{\alpha}{2} \log \log N-X_r(x)}^{\infty}\hspace{-0.1cm}\P(G_N(x)-G_R(x)\in dy) e^{-\frac{\e r(\a/2+\e/2)n-(\a+\e)C-\e'-y)}{3/2(n-r)}}(1+o(1)) dx +Ce^{-\frac{\e r}{32}} \E(W_{\a,N})
\ee
Using the Gaussian tail asymptotics for $G_N(x)-G_R(x)$ together with \eqv(first.12), Equation \eqv(Lisa.105) is bounded above by
\be\Eq(lisa.106)
\E(W_{\a,N})e^{-r\frac{\e^2}{6}+o(r)}.
\ee
This implies the claim of Lemma \thv(Lem.first.3).
\end{proof}

\section{Branching approximation and second moment estimates}

  \subsection{Definition of the increments} 
The goal is to use a branching approximation similar to \cite{ABH17} to compute the necessary second moments. 
To this end, we define for $k\in\N$ and $x\in (0,1)$
\be\Eq(second.1)
Y_k(x)=\sum_{e^{k-1}<\log p \leq e^k}\frac{1}{2\sqrt{p_j}} \left(W_j^{(1)}\cos(x\log p_j)  +W_j^{(2)}\sin(x\log p_j)\right).
\ee
By definition, we have
\be\Eq(second.2)
G_N(x)=\sum_{k=1}^{n} Y_k(x), 
\ee
where for the rest of the section we set $n\equiv \ln\ln N$. 
The increments $Y_k$ are such that
\be\Eq(second.3)
 \rho_k(x,x')\equiv \E(Y_k(x)Y_k(x'))= \sum_{e^{k-1}<\log p \leq e^k}\frac{1}{2p}\cos(|x-x'|\log p_j).
\ee
The covariances can be computed again by the prime number theorem. This is done in Lemma 2.1 in \cite{ABH17}.
It is convenient to state the result to introduce branching point of $x,x'\in (0,1)$ by
\be\Eq(second.4)
x\wedge x'\equiv \lfloor\log \vert x-x'\vert^{-1}\rfloor.
\ee
\begin{lemma}[Lemma 2.1 in \cite{ABH17}] \TH(Lem.sec.1)
For $k\geq 1$ and $x,x'\in(0,1)$ we have
\be\Eq(second.5)
\E(Y_k^2(x))=\frac{1}{2}+ O\left(e^{-c\sqrt{e^k}}\right),
\ee
and
\be\Eq(second.6)
 \rho_k(x,x')= \begin{cases}
\frac{1}{2} +O\left(e^{-2(x\wedge x'-k)}\right)+ O\left(e^{-c\sqrt{e^k}}\right) & \mbox{if } k\leq x\wedge x',\\
O\left(e^{-(k-x\wedge x')}\right) & \mbox{if }k> x\wedge x'
\end{cases} 
\ee
\end{lemma}
There is a fast decoupling between the increments after the branching point where the distribution of $Y_k(x)$ and $Y_k(x')$ is very close to independent Gaussians with mean zero and variance $1/2$. We introduce a parameter $\Delta$ that gives some room before and after the branching point to ensure uniform estimates.

\begin{lemma}\TH(lem.sec.2) Let $\Delta>0$. Let $x,x'\in(0,1)$ and $m>x\wedge x'+\Delta$. Then we have
\be\Eq(second.7)
\P\left(\sum_{k=m+1}^n Y_k(x) \in A, \sum_{k=m+1}^n Y_k(x') \in B\right)=\P\left(\sum_{k=m+1}^n Y_k\in A\right)\P\left(\sum_{k=m+1}^n Y_k\in B\right)\left(1+O(e^{-c\delta}\right),
\ee
where $(Y_i)_{i\in \N}$ are iid Gaussians with mean zero and variance $\sigma^2$.
\end{lemma}
\begin{proof}
As$\left(\sum_{k=m+1}^n Y_k(x),\sum_{k=m+1}^n Y_k(x')\right)$ is a Gaussian process and its covariance is controlled in Lemma \thv(Lem.sec.1) it suffices to compare densities. This follows the same lines starting from Eq. (61) in \cite{ABH17} only that in our setting $\mu=0$.
\end{proof}
Before the branching point we want to show that $Y_k(x)$ and $Y_k(x')$ are almost fully correlated. This is specified in the lemma below.
\begin{lemma}\TH(Lem.sec.3)
Let $\Delta>0$. Let $x,x'\in(0,1)$ and $r<m<x\wedge x'-\Delta$. Then we have
\be\Eq(second.7)
\P\left(\sum_{k=r}^m Y_k(x) \in A, \sum_{k=r}^m Y_k(x') \in B\right)=\P\left(\sum^m_{k=r} Y_k\in A\cap B\right) \left(1+O(e^{-c\D}\right),
\ee
\end{lemma}
\begin{proof}
As $G$ is a Gaussian process this follows from the density estimates in Lemma \thv(Lem.sec.1).
\end{proof}

\subsection{Second moment computation}
The main result of this section is:
\begin{proposition}\TH(prop.sec) There exists $\kappa_\a>0$  such that for $R=o(\log \log N)$ as $N\to \infty$ we have
\be\Eq(prop.sec.1)
\P\left(\left\vert\frac{W_{\a,N}-\E\left(W_{\a,N}\vert \FF_R\right)}{\E\left(W_{\a,N}\right)}\right\vert>c\right)\leq (1+o(1)) Ce^{-k_\a r},
\ee
where $r\equiv\ln\ln R$ and $C>0$ a constant depending on $c$.
\end{proposition}

  To prove Proposition \thv(prop.sec) we essentially need to control the second moment of 
$$
W_{\a,N}^{\leq}=\Leb\Big\{x\in[0,1]: \sum_{j\leq n}Y_j(x)\geq \a n/2; \forall k\in[2r,n]:\;\sum_{j\leq k}Y_j(x)\leq(\alpha+\epsilon)k/2\Big\}.
$$
\begin{remark}
Throughout the proof we restrict our computations to $R$ and $N$ such that $r=\ln\ln R$ and $n=\ln\ln N$ are natural numbers. The general case follows in the same way by considering the last resp. first summands in the representation in \eqv(second.2) of $G_N$ separately. The desired estimates carry over by minor adjustments but would require a more involved notation. To keep the computations that follow as clear as possible and not to burden the reader with heavier notations we restrict ourselves to the case where $r,n\in\mathbb{N}$.
\end{remark}

Indeed, Markov's inequality and Lemma \thv(Lem.first.3) imply
\be\Eq(sec2)
\P\left(\left\vert\frac{W_{\a,N}-\E\left(W_{\a,N}\vert \FF_R\right)}{\E\left(W_{\a,N}\right)}\right\vert>c\right)
\leq\P\left(\frac{\left(W^{\leq}_{\a,N}-\E\left(W^\leq_{\a,N}\ \vert\FF_R\right)\right)^2}{\E\left(W_{\a,N}\right)^2} >c^2/4\right)+ Ce^{-R c(\epsilon)}.
\ee

Clearly, we have
\be
(W^{\leq}_{\a,N})^2
= \Leb^{\times 2}\{x,x'\in[0,1]:\forall y\in\{x,x'\} \ \sum_{k\leq n}Y_k(y)>\frac{\alpha}{2}n, \forall k\in[r,n] \ \sum_{j\leq k}Y_j(y)\leq \frac{\alpha+\epsilon}{2}k \}
\ee
Let $0<\D<r$. 
We divide the right side into four terms depending on the branching point:
$$
(I): x\wedge x'> n-\D \quad (II): r+\D<x\wedge x'\leq n-\D \quad (III): r-\D<x\wedge x'\leq r+\D \quad (IV):x\wedge x'\leq r-\D\ .
$$

The term $(IV)$ is controlled in the following Lemma.
\begin{lemma}\TH(lem.four) For $R=o(\log\log N)$ we have
\be\Eq(four)
\lim_{\D\to\infty}\lim_{N\to\infty}\frac{\E((IV)\vert\FF_R)-\left(\E\left(W^{\leq}_{\a,N}\vert \FF_R\right)\right)^2}{\E\left(W_{\a,N}\right)^2}=0\quad \mbox{ a.s.}
\ee
\end{lemma}
\begin{proof}[Proof of Lemma \thv(lem.four)]
As $x\wedge x'<r-\D$ and by a similar rewriting in \eqv(first.9) we have by Lemma \thv(lem.sec.2) that it is bounded from above by
\bea\Eq(sec.6)
&&\iint_{\{x\wedge x'\leq r-\D\}} \prod_{ y\in\{x,x'\}}\P\left( \sum_{k=r+1}^{n}Y_k> \frac{\alpha}{2} n - G_R(y)-E_R(y)-\e)\Big|\FF_R\right) dx'dx\left(1+O\left(e^{-c\D}\right)\right)\\\nonumber
&\leq& \iint_{[0,1]^2} \prod_{y\in\{x,x'\}}\P\left( \sum_{ k=r+1}^{n}Y_k> \frac{\alpha}{2} n - G_R(y)-E_R(y)-\e)\Big |\FF_R\right) dx'dx \left(1+O\left(e^{-c\D}\right)\right)
\eea
We now compare \eqv(sec.6) with $\left(\E\left(W^{\leq}_{\a,N}\vert \FF_R\right)\right)^2$ which is bounded below by
\be\Eq(sec.7)
\iint_{[0,1]^2} \prod_{y\in\{x,x'\}}\P\left( G_N(y)-G_R(y)> \frac{\alpha}{2}n - G_R(y)-E_R(y)+\e)|\FF_R\right) dx'dx
\ee
for any $\e>0.$
By \thv(Lem.first2) and the Gaussian approximation given in \eqv(Lem.sec.1) the absolute value of the difference of \eqv(sec.6) and \eqv(sec.7) is bounded by
\be\Eq(sec.8)
M_{\a,N}^2 \E(W_{\a,N})^2 e^{-2\e}  \left(1+O\left(e^{-c\D}\right)\right)- M_{\a,N}^2e^{2\a\epsilon}.
\ee
Hence \eqv(sec.8) divided by $\E(W_{\a,N})^2$ converges almost surely to \be M_{\a}^2  e^{-2\a\e} \left(e^{-2\a\e}-e^{2\a\e}+e^{-2\a\e}O\left(e^{-c\D}\right)\right)\ee
for all $\e,\Delta>0$. Note that \eqv(sec.8) converges to zero as $\e \to0 $ and $\D\to\infty$.
\end{proof}

To control the terms $(I)$, $(II)$ and $(III)$, we prove the following lemma.
\begin{lemma}\TH(Lem.sec4)
Let $0<\a<2$. There exists $\kappa_\a>0$  such that for $R=o(\log \log N)$ as $N\to \infty$ we have 
\be\Eq(sec1)
\E\left(  (I)+(II)+(III) \right)
\leq \E\left(W _{\a,N}\right)^2e^{-\kappa_{\alpha} r}.
\ee
\end{lemma}

\begin{proof}[Proof of Lemma \thv(Lem.sec4)]
We bound $\E((I)\vert\FF_R)$ from above by
\be\Eq(second.10)
e^{-n+\D}\int_0^1\P\left(\sum_{k\leq n}Y_k(x)>\frac{\alpha}{2}n\Big\vert\FF_R \right)dx=
e^{-n+\D}\E(W_{\a,N}\vert \FF_R),
\ee
by \eqv(first.8). Hence,
\be\Eq(second.11)
\E((I))\leq \E(W_{\a,N})^2 \frac{e^{-n+\D}}{\E(W_{\a,N})}=E(W_{\a,N})^2 o(1),
\ee
as $\E(W_{\a,N})= c n^{-1/2}e^{-\frac{\alpha^2}{4}n}$ and $0<\alpha<2$.

Next, we turn to $\E((II)\vert \FF_R)$. 
Using that uniformly in $y$ for all $R,N$ large enough $\vert E_N(y)-E_R(y)\vert \leq \epsilon$, we can bound $\E((II)\vert \FF_R)$ from above by
\be\Eq(second.12)
\iint_{\{r+\D\leq x\wedge x'\leq n-\D\}}\P\left(\forall_{y\in\{x,x'\}}\sum_{j=r+1}^nY_j(y)>\frac{\a}{2}n-X_R(y)-\epsilon,\forall_{k\in[r,n] }\sum_{j=r}^kY_j(y)\leq \frac{\alpha+\epsilon}{2} k \Big\vert \FF_R\right) dxdx' 
\ee
Dropping the barrier constraint except at $x\wedge x'-\Delta$ and $x\wedge x'+\Delta$ we can bound the probability in \eqv(second.12) from above by
\be\Eq(second.12.2)
 \P\left(\forall_{y\in\{x,x'\}}\sum_{j=r+1}^nY_j(y)>\frac{\a}{2}n-X_R(y)-\epsilon,\; \forall_{k\in\{x\wedge x'-\Delta,x\wedge x'+\Delta\}} \sum_{j=1}^{k}Y_j(y)\leq \frac{\alpha+\epsilon}{2}k \Big\vert \FF_R\right).
\ee
We evaluate the probability in the integral at a fixed $x\wedge x'=m$, and sum the contributions over $m$ afterwards.
We introduce an extra conditioning. Let $\FF^Y_{k}=\sigma(Y_j,j\leq k)$. We condition on  $\FF^Y_{m+\D}$, slightly after the branching point.
Lemma \thv(lem.sec.2) applied to   \eqv(second.12.2) then yields
\bea\Eq(second.13)
&&  \left(1+e^{-c\D}\right)\E\Bigg(\prod_{y\in\{x,x'\}}\P\left(\sum_{k=m+\D+1}^n Y_k(y)>\frac{\a}{2}n-X_R(y)-\epsilon-\sum_{r<j\leq m+\D}Y_j(y)\Big\vert \FF^Y_{m+\D}\right) \nonumber\\
&&\hspace{3cm}; \forall_{y\in \{x,x'\}, k\in\{m-\D,m+\D\}}\sum_{j\leq k}Y_j(y)\leq \frac{\alpha+\epsilon}{2}k \Big\vert \FF_R\Bigg).
\eea

We distinguish two cases. First, consider the case when for $y=x$ or $y=x'$, \be\frac{\a}{2}n-X_R(y)-\epsilon-\sum_{r<j\leq m+\D}Y_j(y)\leq 0.\ee  Note that due to the barrier in \eqv(second.13) this can only happen jointly with the barrier event if $m\geq \frac{\a}{\a+\e}n-C'\e$ for some constant $C'>0$ independent of $\e$. In this case we bound the probabilities above by one and bound \eqv(second.13) from above by 
\be\Eq(lisa.neu.2)
  \left(1+e^{-c\D}\right) \P\left( \frac{\a}{2}n-X_r(y)-\epsilon-\sum_{r<j\leq m+\D}Y_j(y)\leq 0: \forall_{y\in \{x,x'\}}\sum_{j\leq m+\D}Y_j(y)\leq \frac{\alpha+\epsilon}{2}(m+\D) \Big\vert \FF_R\right).
\ee
As for an upper bound we can drop all constraints in the expectation with respect $x'$ (if $y=x$) and $x$ otherwise, let us assume without loss of generality that $y=x$. We need to distinguish whether $\frac{\a}{2}n-X_R(x)-\epsilon>0$ or not.  On the event $\frac{\a}{2}n-X_R(x)-\epsilon\leq 0$ we bound the expectation in \eqv(lisa.neu.2) by one and obtain that the expectation of \eqv(lisa.neu.2) from above by
\be\Eq(lisa.neu.5)
\P\left(X_R(x)\geq \frac{\a}{2}n-\epsilon\right)\leq \E\left( e^{\a X_R(x) -\a\left(\frac{\a}{2}n-\epsilon\right)}\right)
\ee
by the exponential Chebyshev inequality. Hence, integrating over $x,x'$ in \eqv(lisa.neu.5) we get
\bea\Eq(lisa.neu.6)
e^{-\frac{\a}{\a+\e}n-C'\e}\int_0^1\E\left( e^{\a X_R(x) -\a\left(\frac{\a}{2}n-\epsilon\right)}\right)dx &&\leq \E\left(\int_0^1e^{\a X_r(x)}\right)e^{-\a^2 n/2-\a\e}\nonumber\\
&&\leq Cn\E\left(W_{\a,N}\right)^2e^{-\frac{\a}{\a+\e}n-C'\e} e^{-\a r-\a\e},
\eea
by \eqv(first.12). When $\frac{\a}{2}n-X_R(x)-\epsilon>0$, we bound \eqv(lisa.neu.2) from above using Gaussian tail asymptotics by
\be\Eq(lisa.neu.4)
\left(1+e^{-c\D}\right) \P\left( \sum_{r<j\leq m+\D}Y_j(x)\geq\frac{\a}{2}n-X_R(y)-\epsilon \Big\vert \FF_R\right)\leq \left(1+e^{-c\D}\right)e^{-\frac{\left(\frac{\a}{2}n-X_R(x)-\epsilon\right)^2 }{2\s_r(m+\D)}}.
\ee
The integral of \eqv(lisa.neu.4) with respect to $x$ and $x'$ can be bounded from above by 
\bea\Eq(lisa.neu.5)
&&\left(1+e^{-c\D}\right)\sum_{\frac{\a}{\a+\e}n-C'\e \leq m\leq n-\D}e^{-m} \int_0^1 e^{-\frac{\left(\frac{\a}{2}n-X_r(x)-\epsilon\right)^2 }{2\s_r(m+\D)}} dx\\
&\leq& 
\left(1+e^{-c\D}\right)\sum_{\frac{\a}{\a+\e}n-C'\e \leq m\leq n-\D} e^{-m} \int_0^1 e^{-(\a^2n/4)-\frac{\a^2n(n-2\s_r(m+\D))}{8\s_r(m+\D)} } 
e^{\a \frac{n\e+nX_r(x)}{2\s_r(m+\D)}}dx\nonumber\\
&\leq & \left(1+e^{-c\D}\right)\sum_{\frac{\a}{\a+\e}n-C'\e\leq m\leq n-\D} e^{-m} \int_0^1e^{-(\a^2n/2)-\frac{\a^2(n-2\s_r(m+\D))^2}{8\s_r(m+\D)} +\frac{\a^2}{4}(2\s_r(m+\D)) }e^{\a \frac{n\e+nX_r(x)}{2\s_r(m+\D)}}dx\nonumber
\eea
Using that in the range of summation  in \eqv(lisa.neu.5) $\s_r(m+\D)$ is bounded form above and below by $\frac{1}{2}(m-r)+C$ resp. $\frac{1}{2}(m-r)-C$ , for some constant large enough, we can bound \eqv(lisa.neu.5) from above by
\be\Eq(lisa.neu.6)
\left(1+e^{-c\D}\right)\sum_{\frac{\a}{\a+\e}n-C'\e \leq m\leq n} \int_0^1 e^{-(\a^2n/2)-\frac{\a^2(n- (m-r)-C )^2}{4(m-r+C)}}e^{\left(\frac{\a^2}{4}-1\right)m} e^{\a \frac{n(\e+X_r(x))}{m+\D-r-C}+C \D}dx.
\ee
As $m\geq \frac{\a}{\a+\e}n-C'\e$,   exponantial term in $\epsilon$ bounded by $e^{C\e}$ and as $0<\a<2$ we have that on the one hand $\frac{\a^2}{4}-1<0$ and on the other hand we can choose  together with \eqv(first.12) we can bound the corresponding expectation  in \eqv(lisa.neu.6) from above by
\be
\left(1+e^{-c\D}\right)E(W_{\a,N})^2e^{-c n}e^{cr},
\ee
 for some $c>0$.

Finally, we turn to bound \eqv(second.13) for $\frac{\a}{2}n-X_r(y)-\epsilon-\sum_{r<j\leq m+\D}Y_j(y)\geq 0$ we can bound \eqv(second.13) from above by a Gaussian tail bound and obtain
\bea\Eq(second.14)
&&
\E\Bigg( \frac{(n-m-\D)/2}{2\pi\prod_{y\in\{x,x'\}}(\frac{\a}{2}(n- m-\D-\e)-X_R(y)-\epsilon )
}
 \1_{\forall_{y\in\{x,x'\}, k\in\{m-\D,m+\D\}}\sum_{j\leq k}Y_j(y)\leq \frac{\alpha+\epsilon}{2}k }  \\\nonumber
&& \quad\quad\quad\times
\left.\exp\left({-\sum_{y\in\{x,x'\}}\frac{\left(\frac{\a}{2}n-X_R(y)-\epsilon-\sum_{r<j\leq m+\D}Y_j(y)\right)^2}{(n-m-\D)}  } 
\right)\vert\FF_R\right) 
\eea
Next, we condition on  $\FF^Y_{m-\Delta}$. The  terms depending on $\sum_{m-\D<j\leq m+\D}Y_j$ can be bounded by the moment generating function:
\be\Eq(second.15)
\E\left(e^{C\D \sum_{m-\D<j\leq m+\D}Y_j(x)+Y_j(x')}\right)
\leq e^{C'\D^2}.
\ee
Hence, Equation \eqv(second.14) is bounded above by 
 \be\Eq(second.16)
 e^{C'\D^2} 
\E\left( \frac{(n-m-\D)/2}{2\pi}
\prod_{y\in\{x,x'\}} \1_{ \sum_{j\leq m-\D}Y_j(y)\leq \frac{\alpha+\epsilon}{2}(m-\D) }\frac{\exp\left(-{\frac{\left(\frac{\a}{2}n-X_R(y)-\epsilon-\sum_{j=r+1}^{m-\D}Y_j(y)\right)^2}{(n-m-\D)}  }\right)}{\frac{\a}{2}(n- m-\D-\e)-X_R(y)-\epsilon}
\Big\vert\FF_R\right)
\ee
Using the fact that the variables $Y_j(x)$ and $Y_j(x')$  almost coincide for $j\leq m-\D$ by Lemma \thv(Lem.sec.3), we have that \eqv(second.16) is bounded above by
\be\Eq(second.17)
e^{C'\D^2} \sfrac{(n-m-\D)/2}{2\pi(\frac{\a}{2}(n- m-\D-\e)-X_R(x)-\epsilon )^2 }
\E\left(  \1_{ \sum_{j\leq m-\D}Y_j(x)\leq \frac{\alpha+\epsilon}{2}(m-\D) }e^{-\frac{2\left(\frac{\a}{2}n-X_R(x)-\epsilon-\sum_{j=r+1}^{m-\D}Y_j(x)\right)^2}{(n-m-\D)}  } \Big\vert\FF_R\right) (1+O(e^{-c\D}))
\ee
The expectation in \eqv(second.17) is equal to
\bea\Eq(second.26)
\int_{-\infty}^{\frac{\alpha+\e}{2}  (m-\D) }e^{-\frac{2\left(\frac{\a}{2}n-X_r(x)-\epsilon-z\right)^2}{(n-m-\D)}  } e^{-\frac{z^2}{ (m-\D-r)}} \frac{dz}{\sqrt{\pi  (m-\D-r)}}.
\eea
The integrand with respect to $z$ is minimal for
\be\Eq(second.25)
z^*=\frac{2\left(\frac{\a}{2}n-X_R(x)-\e\right)\left(m-\D-r\right)}{n+m-3\D-2r}.
\ee
When $\frac{\a+\e}{2}m-z^*\ll0$ which is the case when $m<(1-\d)  n$ for some $\d>0$, we can use Gaussian tail asymptotics to bound \eqv(second.26) from above by
\be\Eq(second.27)
\exp\left(-\frac{2\left(\frac{\a}{2}n-X_R(x)-\epsilon-\frac{\a+\e}{2}m\right)^2}{(n-m-\D)}  -\frac{\left(\frac{\a+\e}{2}m\right)^2}{ (m-\D-r)} \right).
\ee
Plugging this bound into \eqv(second.17), summing over $m<(1-\d)n$, and computing the squares in the exponential, we obtain that \eqv(second.17) is bounded from above by
\be\Eq(second.28)
 \sum_{l=r+\D}^{(1-\d)n} e^{\left(\frac{\a^2}{4}-1\right)l} e^{C\D+C\epsilon} \E(W_{\a,N})^2 (1+o(1)),
\ee
If $x\wedge x'<(1-\d)  n$ we can bound the Gaussian integral by one and get that \eqv(second.7) is bounded from above by
\be\Eq(second.29)
e^{-\frac{2\left(\frac{\a}{2}n-X_R(x)-\epsilon\right)^2}{(n-m-\D)}  } e^{+\frac{\left(z^*\right)^2}{ (m-\D-r)}} \frac{dy}{\sqrt{\pi  (m-\D-r)}}e^{-\frac{2\left( z^*\right)^2}{(n-m-\D)}}
\ee
Using \eqv(second.25) we can bound the expectation of \eqv(second.29) for $m>(1-\d)n$ by 
\be\Eq(second.30)
e^{C(\D+\e)} \E\left(W_{\a,N}\right)^2 \E\left(e^\frac{2\left(\frac{\a}{2}n-X_R(x)-\epsilon\right)^2m}{n\left(n+m\right)}\right)
\ee
Plugging this into \eqv(second.17) we can bound the contribution from above
\be\Eq(second.31)
\sum_{m>(1-\d)n} 2^{-m} e^{C(\D^2+\e)} \E\left(W_{\a, N}\right)^2 \E\left(e^\frac{2\left(\frac{\a}{2}n-X_r(x)-\epsilon\right)^2m}{n\left(n+m\right)}\right)
\ee
Noting that $2n-n\d\leq n+m\leq 2n$ the above term can be bounded from above by 
\be\Eq(second.32)
\sum_{m>(1-\d)n} 2^{-m+\frac{2\frac{\a^2}{4} n^2}{n^2(2-\d)l}}e^{C(\D^2+\e)}.
\ee
Note that the exponent in \eqv(second.32) is negative for $\d$ sufficiently small.

Finally, we want to bound $\E((III))$. By Lemma \thv(lem.sec.2) we have similar to \eqv(second.13) that $\E((III))$ is bounded from above by $\left(1+e^{-c\D}\right)$ times
\bea\Eq(second.20)
 &&\sum_{m=r-\D+1}^{r+\D}\iint_{ \{x\wedge x'=m\}}
\E\left(\prod_{y\in\{x,x'\}} \right.\P\left(\sum_{k=m+\D+1}^n Y_k(y)>\frac{\a}{2}n-X_R(y)-\epsilon-\sum_{j=r+1}^{m+\D}Y_j(y)\Big\vert \FF^Y_{m+\D}\right)  \nonumber\\
&&\hspace{6cm};  \left.\forall_{y\in \{x,x'\} }\sum_{j\leq m+\D}Y_j(y)\leq \frac{\alpha+\epsilon}{2}(m+\D) \right)dxdx' 
\eea
  If $ \frac{\a}{2}n-X_R(y)-\epsilon-\sum_{j=r+1}^{m+\D}Y_j(y) > 0$ for $y\in\{x,x'\}$ we  can use  Gaussian tail asymptotics for the probabilities in \eqv(second.20) to bound  the expectation in \eqv(second.20) from above by
\be\Eq(second.21)
\E\left( 
\frac{(n-r)/2}{2\pi\prod_{y\in\{x,x'\}}(\frac{\a}{2}(n- r-2\D-\e)-X_R(y)-\epsilon )} e^{-\frac{\left(\frac{\a}{2}n-X_R(x)-\epsilon-\sum_{j=r+1}^{m+\D}Y_j(x)\right)^2}{(n-m-\D)}  } e^{-\frac{\left(\frac{\a}{2}n-X_R(x')-\epsilon-\sum_{j=r+1}^{m+\D}Y_j(x')\right)^2}{(n-m-\D)} } \right).
\ee
Noticing that the polynomial prefactor is bounded by $C/n$ and otherwise proceeding as  in \eqv(second.15) we can bound \eqv(second.20) from above by
\bea\Eq(second.22)
&&e^{C'\D^2}\sum_{m=r-\D+1}^{r+\D}\iint_{ \{x\wedge x'=m\}}
\E\left( 
 \frac{C}{n}e^{\frac{\left(\frac{\a}{2}n-X_R(x)-\epsilon \right)^2}{(n-m-\D)}  } e^{\frac{\left(\frac{\a}{2}n-X_R(x')-\epsilon \right)^2}{(n-m-\D)} }\vert\FF_R\right)dx'dx\left(1+O\left(e^{-c\D}\right)\right)\nonumber\\
&&\leq e^{C'\D^2-C\D}e^{-r+\D} \E\left(\E\left(W_{\a,N}\right)^2e^{2\a \epsilon} \left(1+O\left(e^{-c\D}\right)\right) \right),
\eea
by \eqv(first.10) for any $\e>0$ and $\D>0$.
 Note first that due to the barrier event in \eqv(second.20) the case $ \frac{\a}{2}n-X_R(y)-\epsilon-\sum_{j=r+1}^{m+\D}Y_j(y) \leq 0$ for at least one $y\in\{x,x'\}$   can be excluded for $m\in\{r-\D,r+\D\}$. 

This completes the control of $(III)$ and hence also the proof of Theorem \thv(Lem.sec4).

\end{proof}

\begin{proof}[Proof of Proposition \thv(prop.sec)]
We bound \eqv(sec2) from above by
\bea\Eq(sec.neu.1)
&&\P\left(\frac{ (I)+(II)+(III)}{\E\left(W_{\a,N}\right)^2} >c^2/8\right)+ \P\left(\frac{\E((IV)\vert\FF_R)-\left(\E\left(W^{\leq}_{\a,N}\vert \FF_R\right)\right)^2}{\E\left(W_{\a,N}\right)^2}>c^2/8\right) +Ce^{-R c(\epsilon)}\\
&&\leq  \frac{8}{c^2}\E\left(\frac{ (I)+(II)+(III)}{\E\left(W_{\a,N}\right)^2}\right)+ \P\left(\frac{\E((IV)\vert\FF_R)-\left(\E\left(W^{\leq}_{\a,N}\vert \FF_R\right)\right)^2}{\E\left(W_{\a,N}\right)^2}>c^2/8\right) +Ce^{-R c(\epsilon)}\nonumber
\eea
where we used Chebyshev's inequality. By Lemma \thv(Lem.sec4) we can bound \eqv(sec2) from above by 
\be\Eq(sec3)
\frac{ 8\E\left(W _{\a,N}\right)^2e^{-\kappa_{\alpha} r}}{c^2\E\left(W_{\a,N}\right)^2}+\P\left(\frac{\E((IV)\vert\FF_R)-\left(\E\left(W^{\leq}_{\a,N}\vert \FF_R\right)\right)^2}{\E\left(W_{\a,N}\right)^2}>c^2/8\right) \leq \frac{4}{c^2}e^{-\kappa_{\alpha} r}+ Ce^{-R c(\epsilon)},
\ee
which yields Proposition \thv(prop.sec) by possibly modifying the constants and noting that $\epsilon$ in \eqv(sec3) is arbitrary (but fixed) as the claim of Lemma \thv(lem.four) holds almost surely in the $N\to\infty$ limit.
\end{proof}

\section{Proof of Theorem \thv(THM.1)}
Finally, we are in the position to prove Theorem \thv(THM.1) using Lemma \thv(Lem.first2) and Proposition \thv(prop.sec).
\begin{proof}[Proof of Theorem \thv(THM.1)]
First, we rewrite
\be\Eq(proof.1)
\frac{W_{\a,N}}{\E\left(W_{\a,N}\right)} 
=\frac{\E\left(W_{\a,N}\vert \FF_R\right)}{\E\left(W_{\a,N}\right)} 
+ \frac{W_{\a,N}-\E\left(W_{\a,N}\vert \FF_R\right)}{\E\left(W_{\a,N}\right)}\ .
\ee
By Proposition \thv(prop.sec) the second summand on the right hand side of \eqv(proof.1) converges to zero in probability when first $N\to\infty$ and then $R\to\infty$.  By Lemma \thv (Lem.first2) the term the first summand on the right hand side of \eqv(proof.1) converges almost surely to $M_\a$ defined in \eqv(model.5). This completes the proof of Theorem \thv(THM.1). 
\end{proof}

\bibliographystyle{abbrv}

\end{document}